\documentclass{article}
\usepackage{graphicx}%
\usepackage{multirow}%
\usepackage{amsmath,amssymb,amsfonts}%
\usepackage{amsthm}%
\usepackage{mathrsfs}%
\usepackage[title]{appendix}%
\usepackage{xcolor}%
\usepackage{textcomp}%
\usepackage{manyfoot}%
\usepackage{booktabs}%
\usepackage{algorithm}%
\usepackage{algorithmicx}%
\usepackage{algpseudocode}%
\usepackage{listings}%
\usepackage{authblk}

\usepackage[numbers]{natbib}
\usepackage{graphicx}
\usepackage{amsmath}
\usepackage{amssymb}
\usepackage{amsthm}
\usepackage[english]{babel}
\usepackage{xcolor}
 
\newcommand{\coeffs}{\mathcal P}
\newcommand{\el}[1]{L_{{#1}}}

\newtheorem{theorem}{Theorem}[section]
\newtheorem*{theorem*}{Theorem}

\newtheorem{proposition}[theorem]{Proposition}%
\newtheorem{lemma}[theorem]{Lemma}
\newtheorem{remark}[theorem]{Remark}

\newtheorem{definition}[theorem]{Definition}

\numberwithin{equation}{section}

\begin{document}

\title{An Unexpected Class of 5+gon-free Line Patterns}
\author[1]{Milena Harned}
\author[2]{Iris Liebman}
\affil[1]{Girls' Angle, \textit{Email:} mdh2192@columbia.edu}
\affil[2]{Girls' Angle, \textit{Email:} iris.liebman@gmail.com}
\date{August 1, 2023}

\maketitle

\begin{abstract}
Let $S$ be a finite subset of ${\mathbb R}^2 \setminus (0,0)$. Generally, one would expect the pattern of lines $Ax + By = 1$, where $(A, B) \in S$ to contain polygons of all shapes and sizes. We show, however, that when $S$ is a rectangular subset of the integer lattice or a closely related set, no polygons with more than 4 sides occur. In the process, we develop a general theorem that explains how to find the next side as one travels around the boundary of a cell.
\end{abstract}


\section{Introduction}
Consider an arbitrary, finite, collection of lines of the form $Ax+By=1$, where $A$ and $B$ are integers, such as the $25$ lines shown in Figure~\ref{example_25}. These lines dissect the plane into regions, the bounded ones of which are convex polygons of various shapes and sizes. For example, the central polygon is a decagon, and there are several pentagons, some of which have been shaded to better illustrate their presence.

\begin{figure}[ht!]
    \centering
    \includegraphics[width=6cm]{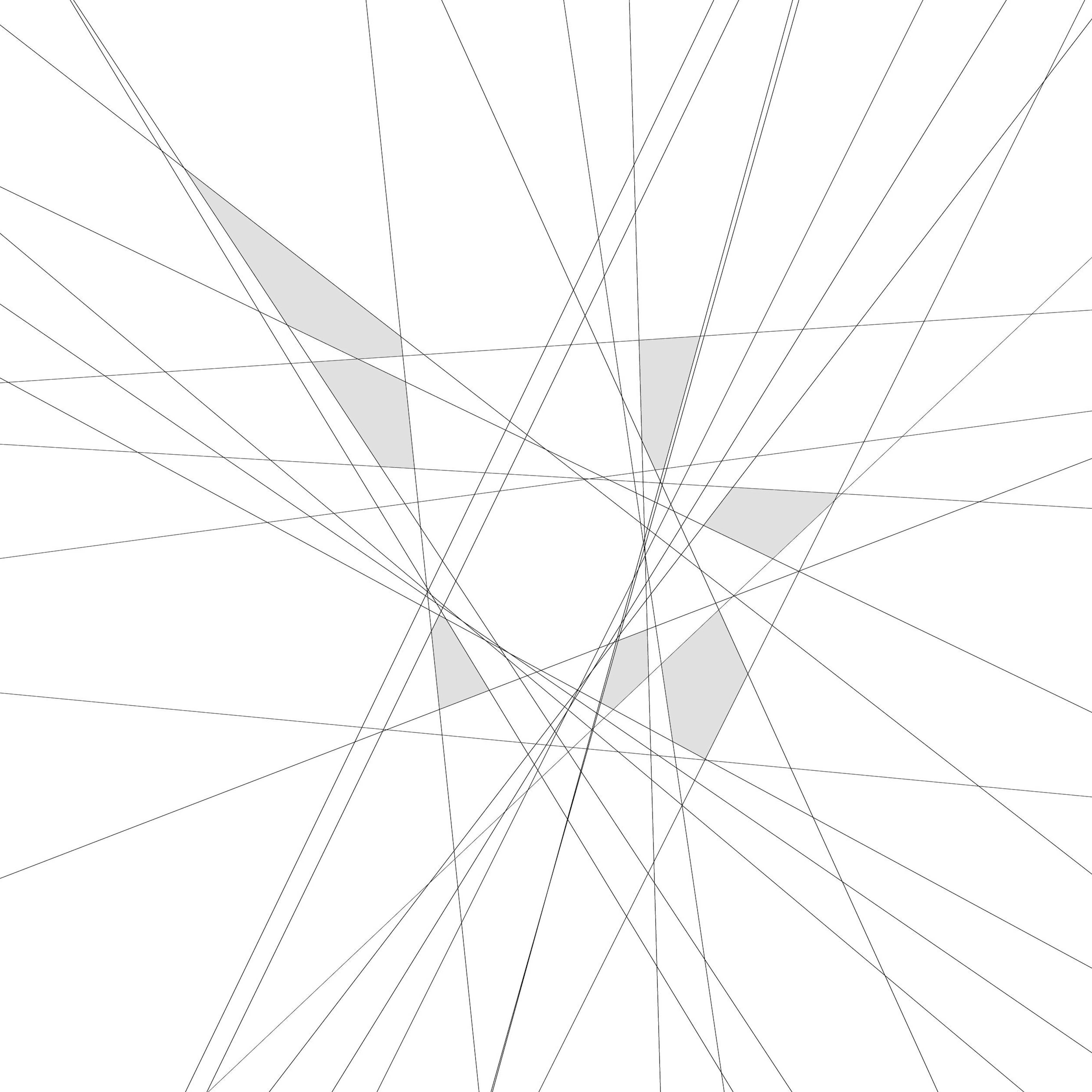}
    \caption{An arbitrary collection of 25 lines of the form $Ax + By = 1$ where $A$ and $B$ are integers. Some pentagons are shaded.}
    \label{example_25}
\end{figure}

Now consider the pattern formed by the $80$ lines of the form $Ax + By = 1$, where $A$ and $B$ are integers between $-4$ and $4$, inclusive. Here every polygon is either a triangle or quadrilateral! Figure~\ref{example_4} shows the first quadrant of this pattern. 

\begin{figure}[ht!]
    \centering
    \includegraphics[width=6cm]{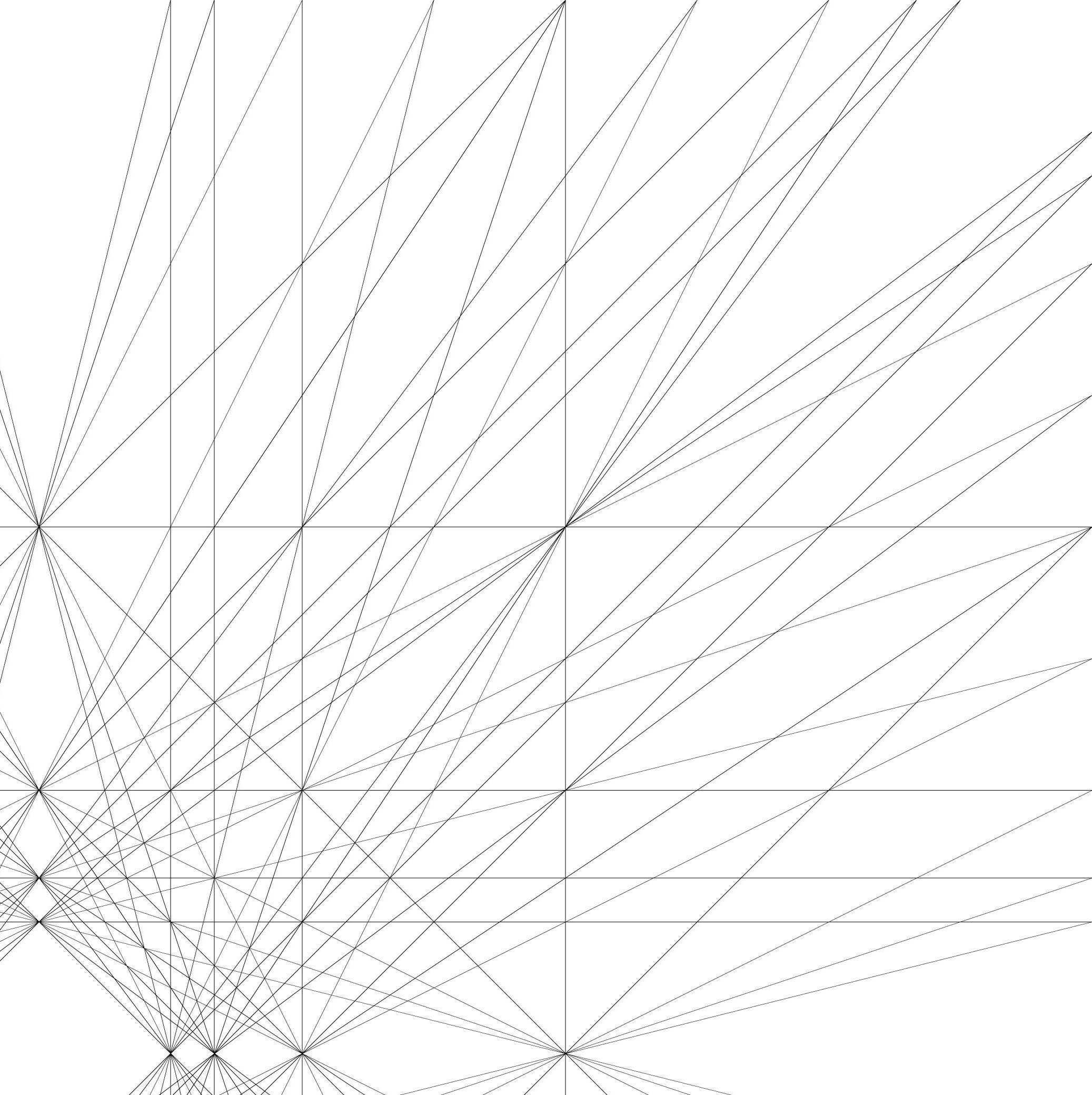}
    \caption{The 56 lines of the form $Ax + By = 1$, where $A$ and $B$ are integers between $-4$ and $4$, inclusive, that are visible in the first quadrant, extended slightly to show the quadrilaterals intersected by the $x$ and $y$ axes.}
    \label{example_4}
\end{figure}

We generalize this observation as follows:
\begin{theorem*} \label{main}
Let $S$ be a rectangular lattice (see Definition~\ref{rectanglelattice}) and consider the pattern of lines of the form $Ax + By = 1$, where $(A, B) \in S$. The bounded regions formed will all be triangles or quadrilaterals.
\end{theorem*}
For a precise statement, see Theorem~\ref{SandT}.

The key observation is that if one walks around the boundary of one of these polygons that does not contain the origin, no three consecutive sides can have the origin on the same side (i.e., to the left or right). Since the origin switches to the opposite side of the sides of a convex polygon exactly twice per lap, no such polygon can have more than $4$ sides.

Theorem~\ref{SandT} is tight in the sense that it requires regularity of the rectangular lattice (see Remark~\ref{pentagon}). We also show that subsets of the lattice obtained by intersecting with a triangle do not necessarily enjoy this property by providing a counterexample in Section~\ref{conclusion}.

To prove Theorem~\ref{SandT}, we developed a general method for determining the next side of a polygon, Theorem~\ref{multiplepoints}, which is one of the main results of this paper. In more detail, suppose that we are travelling clockwise around the boundary of some polygon and $A_1x + B_1y=1$ and $A_2x+B_2y=1$ are lines that contain two of its consecutive sides. This theorem enables us to find the line $Ax+By=1$ which contains the next side of the polygon that we encounter. The proof is somewhat delicate because two intersecting lines can be consecutive sides of $4$ different polygons, so the lines themselves do not supply sufficient information to answer the question.

Line arrangements such as the ones we consider in this paper have been studied for decades. For example, Gr\"unbaum \cite{Grunbaum} provides an early survey and more recently, Lea\~nos, et al., \cite{Leanos2007} studied $5+$gon-free \textit{simple} arrangements. In \cite{Grunbaum}, line arrangements live in the projective plane and are treated combinatorially, with a focus on arrangements where no three lines are concurrent (simple) or where all regions are triangles (simplicial). Here we consider sets of lines that are, in general, neither simple nor simplicial, and we introduce a way to coordinatize these lines. Furthermore, we focus our study on $\mathbb{R}^2$ rather than the projective plane because we use notions of orientation, such as left, right, clockwise, and counterclockwise.
Note, however, that there is no great loss in generality in working in $\mathbb{R}^2$ or using our coordinatization since any finite set of lines can be translated so that none of the lines are the line at infinity or pass through the origin.

\section{Notation}
In this section we state the notation conventions we will use throughout the paper.

We wish to draw a careful distinction between an ordered pair thought of as a pair of coefficients versus as coordinates of a point in the Euclidean plane. For this reason, we define $\coeffs$ to be the plane of coefficients and $\mathbb{R}^2$ to be the Euclidean plane. Thus, if $(A, B) \in \coeffs$, we are thinking of $A$ and $B$ as coefficients in the linear equation $Ax + By = 1$. We think of the line that corresponds to the graph of $Ax + By = 1$ as living in $\mathbb{R}^2$. Specifically, we define $\coeffs$ to be $\mathbb{R}^2 \setminus \{(0, 0)\}$. Then for any $P = (A, B) \in \coeffs$, define $L_P$ to be the line in $\mathbb{R}^2$ defined by the equation $Ax + By = 1$. For any $S \subset \coeffs$, define ${\mathcal L}_S = \{L_P \mid P \in S \}$.

We think of $\coeffs$ as parameterizing the set of lines that do not pass through the origin of $\mathbb{R}^2$. Please note that while $\coeffs$ is closely related to the dual space of $\mathbb{R}^2$, it is not the dual space since we think of points in $\coeffs$ as specific lines, not linear functionals.

We denote by $O$ the origin.

Whenever we speak of ``left'', ``right'', ``clockwise'', or ``counterclockwise'', we mean it from the vantage of one standing on the plane with ``up'' oriented in accordance with the right-hand rule.
Let $P, Q \in \coeffs$. We say that ``Q is to the left of $P$'' if it is in the half plane to the left of $\overleftrightarrow{OP}$. Similarly, we say that ``$Q$ is to the right of $P$'' if it is in the half plane to the right of $\overleftrightarrow{OP}$.

Finally, when we refer to traveling ''to the right'' on a line $L$ that does not contain $O$, we mean to travel in the direction of the points that are to the right of a given point on $L$. Similarly, traveling ``to the left'' on $L$ means to move in the direction of points to the left of a given point on $L$.

\section{General Observations}

\subsection{Lines of the Form $Ax+By=1$}
We first collect some basic facts regarding lines of the form $Ax+By=1$ in a form that we need.
These facts are well-known and we omit their proofs.
For example, Proposition~\ref{slopedistance}(iv) follows immediately from Theorem~3.3.4(a) in \cite{lines}.

\begin{proposition} \label{slopedistance}
Let $P = (A, B) \in \coeffs$. Then
\renewcommand{\labelenumi}{\roman{enumi}}
\begin{enumerate}
\item The line through $O$ and $(A, B)$ in $\mathbb{R}^2$ is perpendicular to $L_P$.\footnote{We state it this way to emphasize that we are thinking of $\coeffs$ and $\mathbb{R}^2$ as different planes.}
\item The $x$-intercept of $L_P$ is $1/A$.
\item The $y$-intercept of $L_P$ is $1/B$.
\item The distance between $(A, B)$ in $\mathbb{R}^2$ and $O$ is the reciprocal of the distance between $L_P$ and $O$. 
\item If $A^2 + B^2 = 1$, then $(A, B)$ in $\mathbb{R}^2$ is on $L_P$.
\item If $A^2 + B^2 > 1$, then the line $L_P$ separates $(A, B)$ in $\mathbb{R}^2$ from $O$.
\item If $A^2 + B^2 < 1$, then $(A, B)$ in $\mathbb{R}^2$ and $O$ are on the same side of $L_P$.
\end{enumerate} 
\end{proposition}

\subsection{Linear Transformations}
Consider a non-degenerate linear transformation $M : \mathbb{R}^2 \to \mathbb{R}^2$.
We denote the transpose of $M$ by $M^t$.
We compute
$$1 = Ax + By =
\big(\begin{smallmatrix}
  A\\
  B
\end{smallmatrix}\big)^t
\big(\begin{smallmatrix}
  x\\
  y
\end{smallmatrix}\big)
=
\big(\begin{smallmatrix}
  A\\
  B
\end{smallmatrix}\big)^t
M^t (M^t)^{-1}
\big(\begin{smallmatrix}
  x\\
  y
\end{smallmatrix}\big)
=
(M \big(\begin{smallmatrix}
  A\\
  B
\end{smallmatrix}\big))^t
(M^t)^{-1}
\big(\begin{smallmatrix}
  x\\
  y
\end{smallmatrix}\big).
$$
Thus, if $(x, y)$ is a point on the line $Ax + By = 1$, then $(M^t)^{-1}
\big(\begin{smallmatrix}
  x\\
  y
\end{smallmatrix}\big)$
is a point on the line
$A^\prime x + B^\prime y = 1$, where
$$
\big(\begin{smallmatrix}
  A^\prime \\
  B^\prime
\end{smallmatrix}\big)
=
M
\big(\begin{smallmatrix}
  A\\
  B
\end{smallmatrix}\big).
$$
In other words, we have

\begin{proposition} \label{lineartransformations}
Let $M : \mathbb{R}^2 \to \mathbb{R}^2$ be a non-degenerate linear transformation and let $S \subset \coeffs$. Then
$${\mathcal L}_{MS} = \{(M^t)^{-1}L_P ~\vert~ P \in S\},$$
where $MS = \{ MC ~\vert~ C \in S \}$.
\end{proposition}

\subsection{Multiple Lines of the Form $Ax+By=1$}

\begin{proposition} \label{parallellines}
A line in $\coeffs$ that passes through the origin corresponds to a family of parallel lines in $\mathbb{R}^2$.
\end{proposition}
\begin{proof}
This follows from Proposition~\ref{slopedistance}(\textit{i}).
\end{proof}

\begin{proposition} \label{concurrent}
Let $m, n \in \mathbb{R}$ be constants not both zero.
Consider the line $l$ of points in the coefficient plane given by $mA + nB = 1$,
Then ${\mathcal L}_{l}$ consists of all lines through $(m, n) \in \mathbb{R}^2$ except for the line through $(m, n)$ that passes through $O$. 
\end{proposition}

\begin{proof}
For any $(A, B) \in l$, we have $mA + nB = 1$ so $(m, n)$ is on every line $L_P$ where $P \in l$.
Conversely, if $L$ is a line that contains $(m, n)$ but not $O$, then we may write the equation of the line $L$ as $Cx+Dy=1$ for some constants $C$ and $D$ such that $Cm + Dn = 1$. Therefore $(C, D) \in l$.
\end{proof}

In the setup of the above proposition, note that the lines $L_P$ rotate around $(m,n)$ in the clockwise direction when $P$ traverses $l$ from left to right.

\begin{proposition} \label{region}
Let $S \subset \coeffs$. The regions created by the lines in ${\mathcal L}_S$ are convex. Furthermore, if $S$ is finite, the bounded regions created by ${\mathcal L}_S$ are convex polygons.
\end{proposition}
\begin{proof}
All the regions formed by the lines in ${\mathcal L}_S$ are intersections of half-planes. Since half-planes are convex, it follows that all regions created by the lines in ${\mathcal L}_S$ are convex.
(For general properties of convex sets, see \cite{convex}.)
Additionally, when $S$ is finite, a bounded region will have a finite number of linear sides, and hence be a polygon. 
\end{proof}

\subsection{Polygons and Points}
Let $S \subset \coeffs$ be finite. From Proposition~\ref{region}, we know that $\mathbb{R}^2$ is partitioned into convex regions, of which the bounded ones are polygons. Let $\Delta$ be one of these polygonal regions, and let $n$ be the number of its sides. 

\begin{figure}[ht!]
    \centering
    \includegraphics[width=11cm]{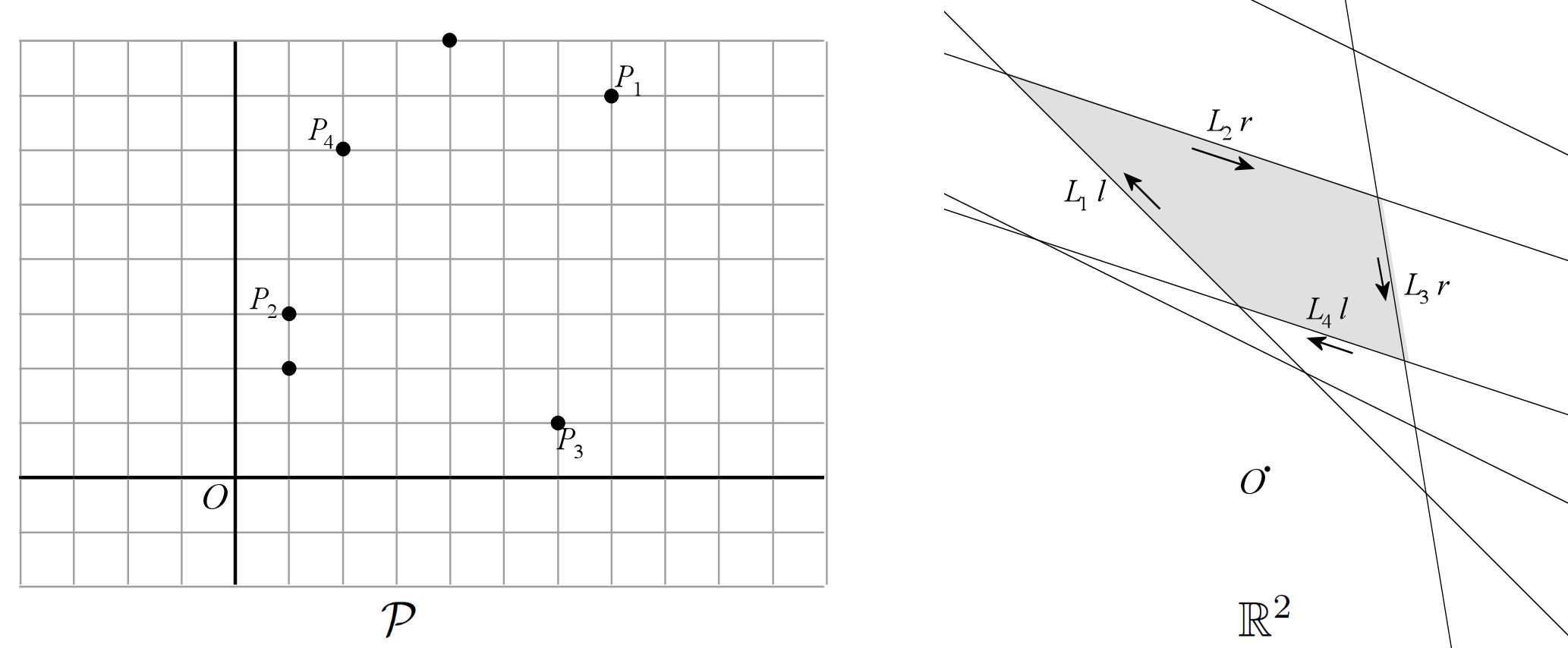}
    \caption{The points $P_k$ corresponding to the lines $L_k$ that bound the shaded region. The values of $D_k$ are shown as well.}
    \label{notation}
\end{figure}

\begin{definition}
Define $L_k$, $k \in \mathbb{Z}/n\mathbb{Z}$ to be the sequence of lines that contain the sides of $\Delta$ as encountered when traveling around $\Delta$ in the clockwise direction (up to translation in the index $k$). Similarly, define the sequence of points in $\coeffs$ that correspond to said $L_k$ to be $P_k$.
\end{definition}

\begin{definition} \label{LPOrientation}
Define the sequence $D_k$, $k \in \mathbb{Z}/n\mathbb{Z}$, to be a sequence of l's and r's (for ``left" and ``right") by declaring $D_k$ to be the side of $L_k$ on which the origin sits when traveling around the boundary of $\Delta$ in the clockwise sense. \end{definition}

See Figure~\ref{notation} for an illustration of these definitions.

\begin{lemma} \label{which_vertical}
If $D_k = D_{k+1}$ then the polygon $\Delta$ and the origin will be contained in a pair of opposite angles formed by $\el{k}$ and $\el{k+1}$. If $D_k \neq D_{k+1}$ then $\Delta$ and the origin will be in adjacent angles formed by $\el{k}$ and $\el{k+1}$.
\end{lemma}

\begin{proof}
Note that as we travel around the boundary of $\Delta$ in the clockwise direction, the region $\Delta$ is always on our right side.

Suppose $D_k = D_{k+1}=r$. Then, the origin is also on the right of \textit{both} sides when walking from $\el{k}$ to $\el{k+1}$. Therefore, both $\Delta$ and the origin are in the same angle formed by $\el{k}$ and $\el{k+1}$.

When $D_k = D_{k+1}=l$, the origin is on the left of {\it both} sides when walking from $\el{k}$ to $\el{k+1}$. Therefore, $\Delta$ and the origin are in opposite angles formed by $\el{k}$ and $\el{k+1}$. 

If $D_k \neq D_{k+1}$, $\Delta$ is on the same side of one of the lines as the origin and on the opposite side of the other line as the origin. Therefore, $\Delta$ and the origin are in adjacent angles formed by $\el{k}$ and $\el{k+1}$. 
\end{proof}

\begin{lemma} \label{consecutivesides}
If $D_k = D_{k+1}$ then $\overleftrightarrow{P_kP_{k+1}} \cap S = \overline{P_kP_{k+1}} \cap S$. If $D_k \neq D_{k+1}$ then $\overline{P_kP_{k+1}} \cap S = \{ P_k, P_{k+1} \}$. 
\end{lemma}

\begin{proof}
Recall from Proposition~\ref{concurrent}, that ${\mathcal L}_{\overleftrightarrow{P_kP_{k+1}}}$ comprises the set of all lines through $\el{k} \cap \el{k+1}$ other than the one which passes through $O$.
The line that contains $O$ is the limit of the lines $L_Z$ as $Z$ moves on $\overleftrightarrow{P_kP_{k+1}}$ toward infinity (in either direction).
Therefore, the union of the lines $L_Z$ for $Z \in \overline{P_kP_{k+1}}$ is the pair of opposite angles created by $\el{k}$ and $\el{k+1}$ which \textit{does not} contain $O$, and the union of the lines $L_Z$ for $Z \in \overleftrightarrow{P_kP_{k+1}} \cap \overline{P_kP_{k+1}}^c$, together with the line that passes through $O$ and $\el{k} \cap \el{k+1}$, is the pair of opposite angles created by $\el{k}$ and $\el{k+1}$ which \textit{does} contain $O$.

If $D_k = D_{k+1}$ then, by Lemma~\ref{which_vertical}, $\Delta$ is in the pair of opposite angles that contains the origin. Therefore, there must not be any point in $S \cap \overleftrightarrow{P_kP_{k+1}} \cap \overline{P_kP_{k+1}}^c$, for lines associated to such points would split $\Delta$.

If $D_k \neq D_{k+1}$, then, by Lemma~\ref{which_vertical}, $\Delta$ is in the pair of opposite angles that does not contain the origin. Therefore, there must not be any point in $\overline{P_kP_{k+1}} \cap S$, other than $P_k$ and $P_{k+1}$, for lines associated to such points would split $\Delta$.
  
\end{proof}

\begin{proposition} \label{verticalangles}
Let $x \in {\mathbb{R}^2}$ be any point not in $\Delta$ or on the extension of any of the sides of $\Delta$.
Then there exist exactly two vertices of $\Delta$ where $\Delta$ and $x$ are in adjacent angles formed by the extensions of the sides intersecting at that vertex. 
\end{proposition}
\begin{proof}
Let us first draw the lines connecting $x$ with each vertex of $\Delta$. Consider the vertices $R_i$ and $R_j$ of $\Delta$ such that the angle formed by the rays $\overrightarrow{xR_i}$ and $\overrightarrow{xR_j}$ is maximized. Since this angle is maximized, $xR_i$ and $xR_j$ do not intersect the interior of $\Delta$, so no pair of opposite angles formed by sides meeting at either $R_i$ or $R_j$ will contain both $\Delta$ and $x$.

The lines connecting $x$ with the remaining vertices of the polygon will be contained within the angle formed by $\overrightarrow{xR_i}$ and $\overrightarrow{xR_j}$, thus, these lines must go through the interior of $\Delta$. For any vertex $R_k$ where $\overrightarrow{xR_k}$ goes through the interior of $\Delta$, the pair of opposite angles formed by the sides of $\Delta$ with a vertex at $R_k$ will contain both $\Delta$ and all points on $\overrightarrow{xR_k}$, including $x$. 
\end{proof}

\begin{lemma} \label{numberorient}
If we arrange the values of the sequence of $D_1$, \dots, $D_n$ in order around a circle, there will be two places where the adjacent values differ, unless $\Delta$ contains the origin, in which case all the values will be r. 
\end{lemma}

\begin{proof}
Assume $\Delta$ does not contain the origin. Imagine rotating $\el{k}$ clockwise to $\el{k+1}$ around $\el{k} \cap \el{k+1}$. By Lemma~\ref{which_vertical}, note that as we rotate the line, it will cross over the origin if and only if $D_k \neq D_{k+1}$. This occurs if and only if $\el{k} \cap \el{k+1}$ is equal to $R_i$ or $R_j$, where $R_i$ and $R_j$ are as defined in the proof of Proposition~\ref{verticalangles}. Thus, there will be only two places where the values of $D_k$ change.

When $\Delta$ contains the origin, there is no place where the values can change, since all rays from the origin through a vertex of $\Delta$ go through the interior of the polygon; and, indeed, when $\Delta$ contains the origin, $D_k = r$ for all $k$.
\end{proof}

\begin{proposition}\label{valueofd}
If $P_{k+1}$ is to the left of $P_k$ then $D_{k+1} \neq D_k$. If $P_{k+1}$ is to the right of $P_k$ then $D_{k+1} = D_k$.
\end{proposition}

\begin{proof}
Consider the line $\overleftrightarrow{P_{k}P_{k+1}}$.
As a point $P \in \coeffs$ travels from $P_k$ to the right along $\overleftrightarrow{P_{k}P_{k+1}}$, the line $L_P$ rotates clockwise. If $P_{k+1}$ is hit before $P$ wraps around to the opposite side, then $D_{k+1}$ will stay equal to $D_k$. As $P$ wraps around to the opposite side, $L_P$ in $\mathbb{R}^2$ passes over the origin. Therefore, if $P_{k+1}$ is to the left of $P_k$, then $D_{k+1} \neq D_k$. 
\end{proof}

\begin{figure}[ht!]
    \centering
    \includegraphics[width=11cm]{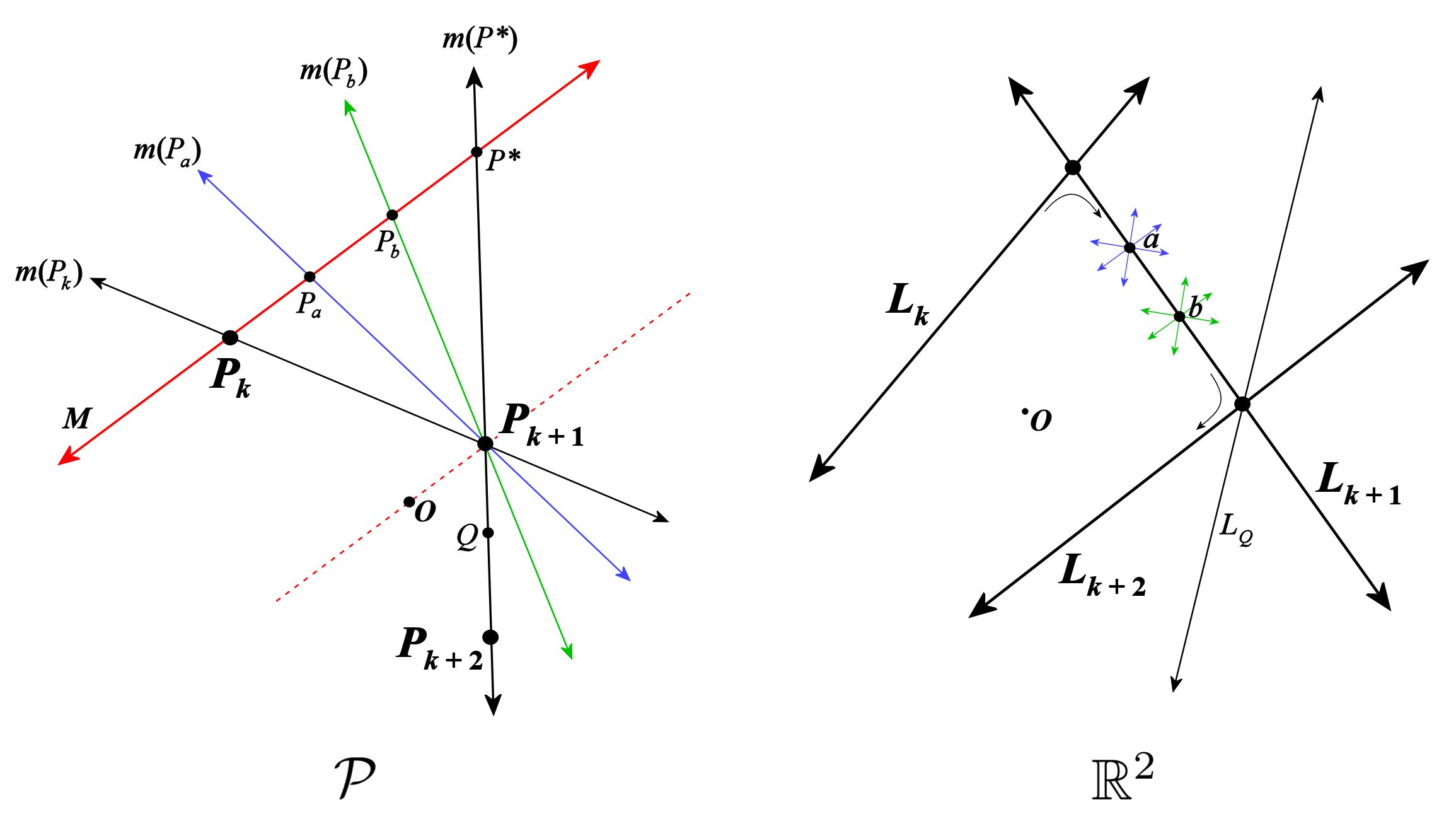}
    \caption{Illustration of our strategy to find the next side around a polygonal region (for the case $D_k=D_{k+1}=r$).  Points $a$ and $b$ are the points of concurrency of $\mathcal{L}_{m(P_a)}$ and $\mathcal{L}_{m(P_b)}$, respectively. (Not drawn to scale.)}
    \label{nextSideIllustrated}
\end{figure}

We now explain how to find $P_{k+2}$, given $\Delta$, $P_k$, and $P_{k+1}$.

For this discussion, please refer to Figure~\ref{nextSideIllustrated}. Our strategy for finding $P_{k+2}$ is to travel along $\el{k+1}$ to the next vertex of $\Delta$ (i.e., from $\el{k} \cap \el{k+1}$ to $\el{k+1} \cap \el{k+2}$). As we travel, we take note of the set of points in $\coeffs$ that corresponds to the lines through each position on $\el{k+1}$ that we pass through. We stop when this set of points contains at least one point in $S$ other than $P_{k+1}$. We then ascertain which of these points of $S$ corresponds to $P_{k+2}$.

By Proposition~\ref{concurrent}, we know that for any line $K \subset \coeffs$ through $P_{k+1}$, ${\mathcal L}_K$ corresponds to a set of lines in $\mathbb{R}^2$ that concur at a point on $\el{k+1}$. If $K = \overleftrightarrow{P_kP_{k+1}}$, then the point of concurrency of the lines in ${\mathcal L}_K$ is $\el{k} \cap \el{k+1}$. As $K$ rotates around $P_{k+1}$ away from $\overleftrightarrow{P_kP_{k+1}}$, the point of concurrency moves along $\el{k+1}$ away from $\el{k} \cap \el{k+1}$. As $K$ rotates toward $\overleftrightarrow{OP_{k+1}}$ from either side, the point of concurrency moves to infinity.

To be more precise, let $M$ be the line through $P_k$ that is parallel to $\overleftrightarrow{OP_{k+1}}$.
Note that $M$ does not contain $O$ because $\el{k}$ and $\el{k+1}$ intersect and so cannot be parallel, which means $\overleftrightarrow{P_kP_{k+1}}$ does not contain the origin (See Proposition~\ref{parallellines}).
Therefore, we can define a map $m : M \to \{ \text{lines in $\coeffs$ that contain $P_{k+1}$} \}$ by sending each $P \in M$ to $m(P) = \overleftrightarrow{PP_{k+1}}$.
Note that $m(P)$ parametrizes all the lines in $\coeffs$ that pass through $P_{k+1}$. As $P$ moves right along $M$, $m(P)$ rotates clockwise about $P_{k+1}$ and as $P$ tends to infinity in either direction, $m(P)$ tends to the line $\overleftrightarrow{OP_{k+1}}$.
Thus, as $P$ moves away from $P_k$, the point of concurrency of the lines in ${\mathcal L}_{m(P)}$ moves away from $\el{k} \cap \el{k+1}$, and as $P$ tends to infinity in either direction, the point of concurrency tends to infinity.

Starting at $P_k$, which way should $P$ move along $M$ so that the corresponding point of concurrency of the lines in ${\mathcal L}_{m(P)}$ moves toward the next vertex of $\Delta$?
The answer, as we shall show, is that $P$ should move along $M$ in the direction given by $D_{k+1}$.
We first note that by Definition~\ref{LPOrientation}, $\el{k+2}$ intersects $\el{k+1}$ on the origin side or the non-origin side of $\el{k}$ depending on whether $D_k$ is $r$ or $l$, respectively.
Therefore, $P$ must be moved so that the point of concurrency of the lines in ${\mathcal L}_{m(P)}$ are on the origin or the non-origin side of $\el{k}$ depending on whether $D_k$ is $r$ or $l$, respectively.
We can determine which side of $L_k$ the point of concurrency is on by looking at which line, parallel to $L_k$, passes through the point of concurrency, that is, looking at $L_{m(P) \cap \overleftrightarrow{OP_k}}$ (unless $m(P)$ is parallel to $\overleftrightarrow{OP_k}$, in which case the point of concurrency is where the line parallel to $\el{k}$ and through $O$ intersects $\el{k+1}$).
Therefore, when $P \neq P_k$, the point of concurrency is in the non-origin side of $\el{k}$ if $m(P) \cap \overleftrightarrow{OP_k} \in \overline{OP_k}$, and on the origin side of $\el{k}$ otherwise.
If $P_{k+1}$ is to the right of $P_k$, then $m(P) \cap \overleftrightarrow{OP_k} \in \overline{OP_k}$ if and only if $P$ is to the left of $P_k$.
If $P_{k+1}$ is to the left of $P_k$, then $m(P) \cap \overleftrightarrow{OP_k} \in \overline{OP_k}$ if and only if $P$ is to the right of $P_k$.
Since $m(P)$ never crosses over the origin, rightward motion along $m(P)$ corresponds to the direction from $P$ to $P_{k+1}$ if $P_{k+1}$ is to the right of $P_k$ and the direction from $P_{k+1}$ to $P$ if $P_{k+1}$ is to the left of $P_k$.

With these observations in mind, suppose $(D_k, D_{k+1}) = (r, r)$.
Then we must move $P$ so that the point of concurrency moves into the origin side of $L_k$.This means that we do not want $m(P) \cap \overleftrightarrow{OP_k} \in \overline{OP_k}$. Now $P_{k+1}$ cannot be to the left of $P_k$ because with $D_{k+1}=D_k=r$, we would have to be able to reach $P_{k+1}$ from $P_k$ by moving to the right. So $P_{k+1}$ is right of $P_k$, which means $P$ must move to the right from $P_k$, which corresponds to $D_{k+1}$.
A similar argument shows that $P$ moves from $P_k$ in the direction given by $D_{k+1}$ in all cases (see Figure~\ref{possibilities}).

We move $P$ until $m(P)$ first touches points of $S$ other than those on $\overleftrightarrow{P_kP_{k+1}}$. Let $P^*$ be the value of $P$ when this occurs.

\begin{figure}
\includegraphics[width=12cm]{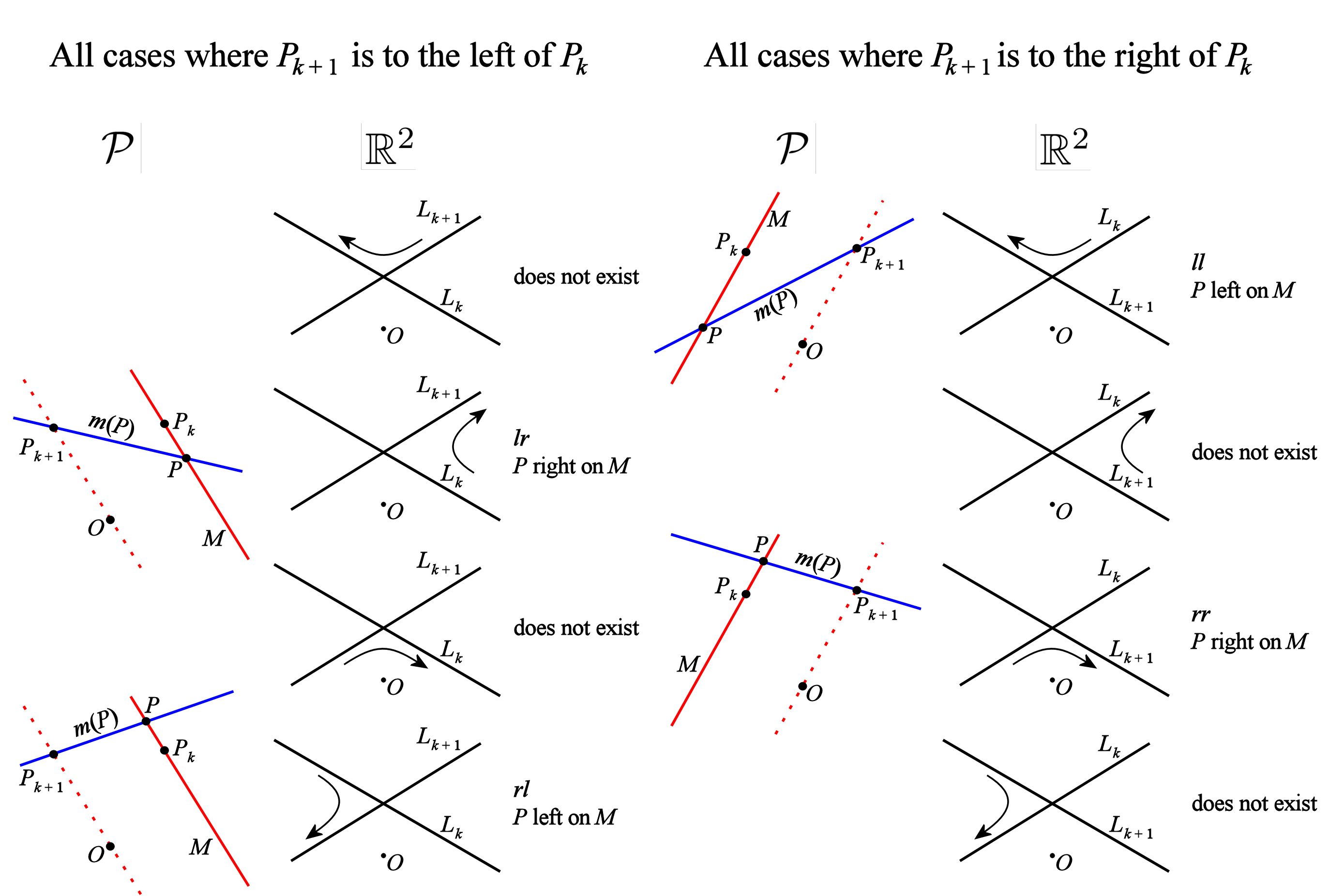}
\centering
\caption{All possible cases for going clockwise around the polygon from a side on $L_k$ to the next side on $L_{k+1}$. Next to each case we list $D_{k}$ and $D_{k+1}$ and indicate which direction $P$ must move to find the next vertex of $\Delta$. See the passage following the proof of Proposition~\ref{valueofd} for an explanation of this and the notation.}
\label{possibilities}
\end{figure}

Since $S$ is a finite set, there are finitely many points in $m(P^*) \cap S$. Let $Q_0$, $Q_1$, $Q_2$, $Q_3$, \dots, $Q_q$ be the points of $m(P^*) \cap S$ with $Q_0 = P_{k+1}$ and the $Q_i$ indexed in the order that the points appear on $m(P^*)$ starting at $Q_0$ and moving away from the direction of $P^*$, wrapping around to the opposite end of $m(P^*)$ if necessary.

Which of the points $Q_i$ corresponds to $P_{k+2}$? Since we are progressing around the boundary of $\Delta$ in the clockwise direction, when we reach $\el{k+1} \cap \el{k+2}$, the next side will be on the line $L_{Q_k}$ in which $Q_{k}$ is the last of the $Q_k$'s, $k>0$, that we encounter as we rotate away from $\el{k+1}$ in the clockwise direction.

Suppose $D_k=D_{k+1}$. In this case $P_{k+1}$ is to the right of $P_{k}$, so rightward motion along $m(P^*)$ corresponds to walking along $m(P^*)$ from $P_{k+1}$ away from $P^*$. Therefore, clockwise rotation about $\el{k+1} \cap \el{k+2}$ corresponds to walking along $m(P^*)$ away from $P^*$. Therefore, the next side will be $Q_q$.

Suppose $D_k \neq D_{k+1}$. In this case $P_{k+1}$ is to the left of $P_{k}$, so rightward motion along $m(P^*)$ corresponds to walking along $m(P^*)$ from $P_{k+1}$ toward $P^*$. Therefore, clockwise rotation about $X$ corresponds to walking along $m(P^*)$ toward $P^*$. Therefore the next side will be $Q_1$.

To summarize,

\begin{theorem}\label{multiplepoints}
In the notation set up in the preceding discussion, $P_{k+2} = Q_q$ if $D_k=D_{k+1}$ and $P_{k+2} = Q_1$ if $D_k \neq D_{k+1}$. 
\end{theorem}

\subsection{The Region Containing $O$.}

Let $S \subset \coeffs$ be a finite set.
In this section, we identify the points in $S$ that correspond to the sides of the region containing $O$.

\begin{definition} \label{sidesaroundo}
Let $O_k \in \coeffs$ be the points in $S$ that correspond to the sides of the region containing $O$ as encountered in the clockwise order. If the region containing $O$ is bounded, $O_k$ is defined up to translation of the indices. Otherwise, $O_k$ is indexed so that the unbounded sides correspond to the first and last terms of the sequence.
\end{definition}

The exact determination of $O_k$ depends on whether the region containing $O$ is bounded or not and on the dimension of the convex hull of $S$.
We shall now go through the various cases and explain precisely what $O_k$ is, and then prove that our prescription is correct in Proposition~\ref{regioncontainingo}.
Let $C$ be the convex hull of $S$. (Note that the vertices of $C$ are among the points of $S$.) 
We consider six cases:
\renewcommand{\labelenumi}{\roman{enumi}}
\begin{enumerate}
    \item $C$ is not a polygon.
    \begin{enumerate}
        \item $S = \{P\}$. Set $\{O_k\}$ to be the one-term sequence with $O_1 = P$.
        
        \item $C$ is a line segment of nonzero length and the line containing $C$ also contains the origin. If the origin is not interior to $C$, set $\{O_k\}$ to be the one-term sequence with $O_1$ being the point of $S$ farthest from the origin. If the origin is interior to $C$, set $\{O_k\}$ to be the two-term sequence with $O_1$ and $O_2$ being the endpoints of $C$ in either order.
        
        \item $C$ is a line segment of nonzero length and the line containing $C$ does not contain the origin. Set $\{O_k\}$ to be the two-term sequence consisting of the endpoints of $C$ labeled so that $O_2$ is to the right of $O_1$.
    \end{enumerate}
    \item $C$ is a polygon.
    \begin{enumerate}
        \item The interior of $C$ contains the origin (see Figure~\ref{ccontainso}). Set $O_1$, $O_2$, $O_3$, \dots, $O_N$ to be the vertices of $C$ in clockwise order, starting at an arbitrary vertex of $C$. 
        
        \item $C$ does not contain the origin. Then there are exactly two rays emanating from the origin that pass through the boundary of $C$ but not its interior. Label these rays $r_1$ and $r_2$ so that the points of $r_2 \setminus \{(0,0)\}$ are to the right of the points of $r_1 \setminus \{(0,0)\}$. Between $r_1$ and $r_2$, the boundary of $C$ contains two paths that both have one point on $r_1$ and one point on $r_2$. (Note that the union of these two paths may not comprise the entire boundary of $C$ if $r_1$ or $r_2$ intersects the boundary of $C$ in a line segment.) Each ray $\Delta$ from the origin that intersects the interior of $C$ intersects the boundary of $C$ in two points (because $C$ is convex), one farther from the origin than the other. Call the one of the two paths that contains the points on these rays farther from the origin the ``outer'' path. Set $O_1$ to be the point of $S$ which is the point on this outer path on $r_1$. As we travel from $O_1$ along the outer path, by construction, we will be going clockwise around $C$. Let $O_2$, $O_3$, \dots, $O_N$ be the vertices of $C$ as we encounter them, with $O_N$ being the point of the outer path on $r_2$. (Note that $C$ may have more than $N$ vertices.)
        
        \item  The origin is on the boundary of $C$. Set $O_1$, $O_2$, $O_3$, \dots, $O_N$ to be the vertices of $C$ as they are encountered if one travels once around the boundary of $C$ in the clockwise direction starting at the origin. 
    \end{enumerate}
\end{enumerate}

\begin{figure} 
    \centering
    \includegraphics[scale=0.5]{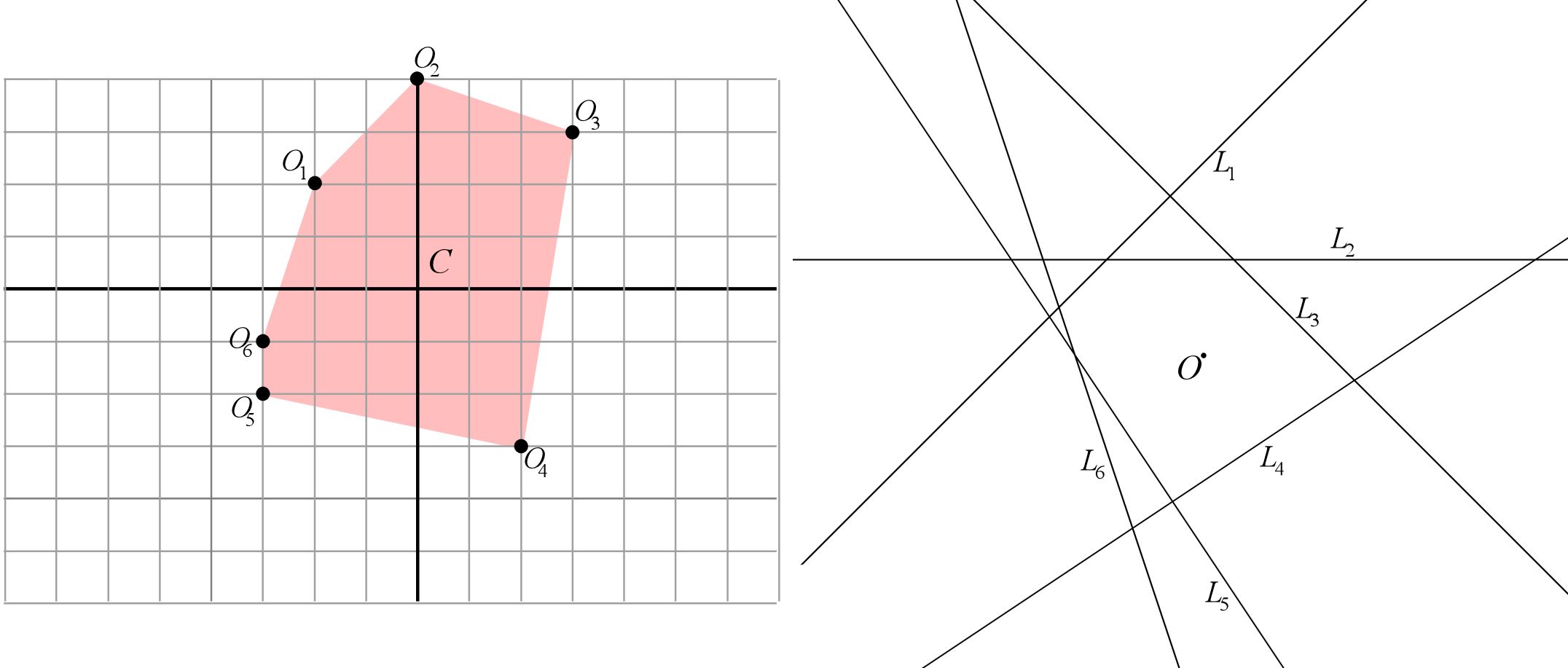}
    \caption{With notation as in Definition~\ref{sidesaroundo}, $S$ consists of the indicated $6$ points and $C$ is the convex hull of $S$. Here, $C$ contains the origin. Note that the region that contains the origin in the pattern formed by the lines of ${\mathcal L}_S$ is bounded.}
    \label{ccontainso}
\end{figure}

\begin{figure} 
    \centering
    \includegraphics[scale=.5]{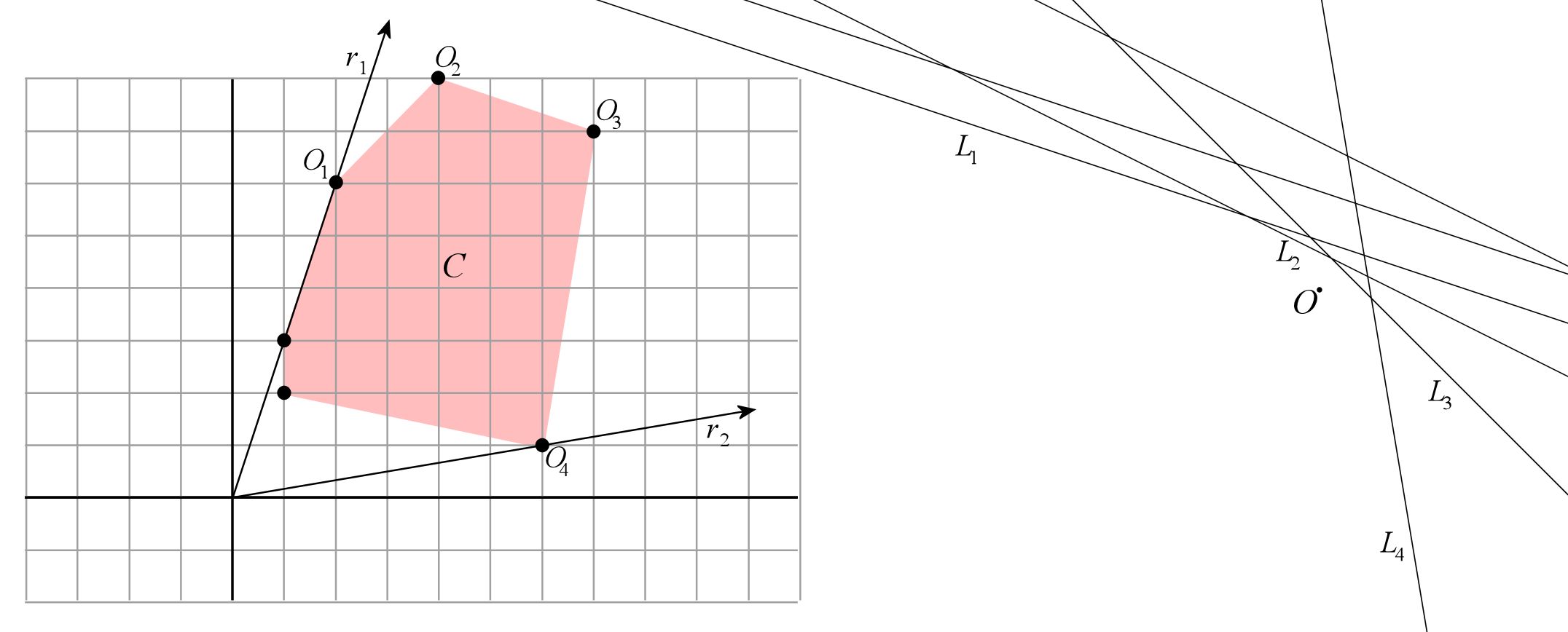}
    \caption{With notation as in Definition~\ref{sidesaroundo}, $S$ consists of the indicated $6$ points and $C$ is the convex hull of $S$. Here, $C$ does not contain the origin. Note that the region that contains the origin in the pattern formed by the lines of ${\mathcal L}_S$ is unbounded.}
    \label{notccontainso}
\end{figure}

\begin{proposition}\label{regioncontainingo}
Let $S \subset \coeffs$ be a finite set. Let $N$ be the length of the finite sequence $O_k$, as defined in Definition~\ref{sidesaroundo}. Let $\Delta$ be the region that contains $O$. The boundary of $\Delta$ is precisely given by the lines $L_{O_k}$, $1 \le k \le N$.
\end{proposition}

\begin{proof}
We use the notation and case numbering from the discussion just after Definition~\ref{sidesaroundo}.

The cases i(a-c) may be directly verified. So assume that the convex hull of $S$ is a nondegenerate convex polygon.

Case ii(a). Let us regard the indices of the sequence $\{O_k\}$ modulo $N$. Note that $\overleftrightarrow{O_kO_{k+1}} \cap S = \overline{O_kO_{k+1}} \cap S$. Also, $O_{k+1}$ is to the right of $O_k$, by construction. Therefore, if one travels from $L_{O_k}$ onto $L_{O_{k+1}}$ in the clockwise direction, one will be walking around the boundary of some region $\Delta$. Furthermore, $O$ will be on the right as one walks along this part of the boundary of $\Delta$. By Theorem~\ref{multiplepoints}, we find that the next side of $\Delta$ encountered will indeed by $L_{O_{k+2}}$. Thus, the $\{O_k\}$ are the sides of a bounded region for which $O$ is always to the right as one travels around its boundary in the clockwise direction. Since clockwise travel around the boundary of any bounded convex region that does not contain $O$ must involve $O$ switching from right to left or left to right exactly twice, we conclude that $\Delta$ must be the region that contains $O$. 

Case ii(b-c). When $C$ does not contain the origin, the sequence $\{O_k\}$ must still form the boundary of some region $\Delta$ for which $O$ is always on the right as one travels around $\Delta$ in the clockwise direction. However, by construction, when one reaches $L_{O_N}$ from $L_{O_{N-1}}$, one will walk off to infinity without encountering any further intersection with lines of ${\mathcal L}_S$ (because, as in the discussion preceding Theorem~\ref{multiplepoints}, we would rotate $\overleftrightarrow{O_{N-1}O_N}$ about $O_{N}$ and pass over the origin before meeting another point in $S$). Similarly, by reversing the construction in the discussion preceding Theorem~\ref{multiplepoints}, we see that $L_{O_1}$ also has no further intersections with lines in ${\mathcal L}_S$ when one travels counterclockwise along it from the intersection of $L_{O_2}$ and $L_{O_1}$. Thus, the $\{O_k\}$ must correspond to the lines that form the boundary of the unbounded region which contains $O$ (because if $O$ were not in the region, $O$ must still switch to one's left at some point as one travels clockwise about the boundary of the region, even though the region is unbounded).
\end{proof}

\section{Rectangular Lattices}\label{rectangularlattices}
We define a ``rectangular lattice'' as follows.
\begin{definition}\label{rectanglelattice}
Let $(a, b) \in \coeffs$. Let $\Delta_x, \Delta_y \in {\mathbb R}_{>0}$. Let $N, M \in {\mathbb Z}_{\ge 0}$.
Assume in addition that the three corners $(a, b + M\Delta_y)$, $(a + N\Delta_x, b)$, and $(a + N\Delta_y, b + M\Delta_x)$ are all in $\coeffs$.
We define a rectangular lattice to be a subset $S$ of $\coeffs$ of the form: 
$$S \equiv \{(a + k\Delta_x, b + j\Delta_y) ~\vert~ 0 \le k \le N, 0 \le j \le M, k, j \in {\mathbb Z}_{\ge 0} \} \cap \coeffs.$$
\end{definition}
(The purpose of intersecting with $\coeffs$ is merely to eliminate the possibility of including the origin.) We use this definition of $S$ throughout this section.

We call a point of the form $(a + k\Delta_x, b + j\Delta_y)$ where $k = 0$, $k=N$, $j=0$, or $j=M$ a \textbf{boundary point} of $S$ and refer to the union of these points as the \textbf{boundary} of $S$.

\begin{lemma}\label{RRL}
Let $S \subset \coeffs$ be a rectangular lattice and let $\Delta$ be a bounded polygonal region created by the lines of ${\mathcal L}_S$ which does not contain $O$. Then no three consecutive $D_k$'s are equal.
\end{lemma}

\begin{proof}
Assume $D_k = D_{k+1}$. We shall prove that $D_{k+1} \neq D_{k+2}$. We split into two cases, depending on whether $D_k=D_{k+1}$ is $r$ or $l$. We begin with the case $D_k=D_{k+1}=r$. By Proposition~\ref{valueofd}, $P_{k+1}$ is to the right of $P_k$.

By Lemma~\ref{consecutivesides}, $\overleftrightarrow{P_kP_{k+1}} \cap S = \overline{P_kP_{k+1}} \cap S$. Also, because $L_k$ and $L_{k+1}$ are consecutive sides of $\Delta$, they cannot be parallel, therefore the origin of the coefficient plane is not on $\overleftrightarrow{P_{k}P_{k+1}}$.

We use the notation as setup in the discussion of how to find the next side preceding Theorem~\ref{multiplepoints}. Thus, $m(P) = \overleftrightarrow{PP_{k+1}}$ where $P$ is on the line $M$ which passes through $P_k$ and is parallel to $\overrightarrow{OP_{k+1}}$.
According to our discussion there, since $D_{k+1} = r$, the next side around $\Delta$ is found by moving $P$ to the right along $M$ from $P_k$. Let $m$ be the first line $m(P)$ that hits another point in $S$ as $P$ move right from $P_k$. (In the notation preceding Theorem~\ref{multiplepoints}, we have $m = m(P^*)$.)
Because $\Delta$ is bounded, $m(P)$ will not pass through the origin as $P$ travels from $P_{k+1}$ to $P^*$. 
Let $m_r$ be the portion of $m$ to the right of $P_{k+1}$ and let $m_l$ be the portion of $m$ to the left of $P_{k+1}$.

We will show that $D_{k+2} = l$.

Let $P_k = (x_k, y_k)$ and $P_{k+1} = (x_{k+1}, y_{k+1})$.
Let $P_k' = (x_k', y_k')$ be the reflection of $P_k$ over $P_{k+1}$ (i.e., $x_k' = 2x_{k+1}-x_k$ and $y_k' = 2y_{k+1}-y_k$). Note that because $\overleftrightarrow{P_kP_{k+1}} \cap S = \overline{P_kP_{k+1}} \cap S$, $P_k'$ is not in $S$. Also, $P_k' \neq O$, since $\el{k}$ and $\el{k+1}$ are not parallel.

We now show that there exists points of $S$ on $m_l$ by contradiction. Suppose there are no such points. Then there must exist a point $(x, y)$ in $m_r \cap S$.

Let us first consider the case where $y_{k+1} < y_k$ and $x_{k+1} > x_k$. A representation of this is shown in Figure~\ref{xyproof}.

We claim that $(x_{k+1}, y_k) \in S$. By definition of rectangular lattice, the only way this would not be the case is if $(x_{k+1}, y_k) = O$, but this is not the case since $P_{k+1}$ is to the right of $P_k$.
Therefore, as $P$ moves from $P_k$ to $P^*$, $m(P)$ will not rotate beyond the vertical. If $m$ is vertical, then $(x_{k+1}, y_k) \in m_l \cap S$. So assume $m$ is not vertical and let $(x, y)$ be a point on $m_r$.  Then, $x > x_{k+1}$. 

Note that as $P$ moves from $P_{k}$ to $P^*$, $m(P)$ rotates clockwise because $P^*$ is to the right of $P_{k}$. Since $m$ was obtained by a clockwise rotation of $m(P)$, we cannot have both $x \ge x_k'$ and $y \ge y_k'$. If $(x, y)$ satisfies $x \ge x_k'$ and $y \le y_k'$, it is not in $S$, for otherwise, $P_k' \in S$, a contradiction. If $x < x_k'$ and $y < y_k'$, then $(x, y_k')$ is in $S$. When $m(P)$ rotates clockwise, $(x, y_k')$ would be hit before $(x, y)$, therefore $(x, y)$ cannot be on $m_r$. 

Therefore any point $(x, y)$ in $S$ on $m_r$ other than $P_{k+1}$ must satisfy $y \ge y_k'$ and $x \le x_k'$ (with not both being equalities). But the reflection of these points over $P_{k+1}$ are points on $m_l$ that are also in $S$, a contradiction.

The cases where $x_{k+1} < x_k$ and $y_{k+1} < y_k$, or where $x_{k+1} < x_k$ and $y_{k+1} > y_k$, or where $x_{k+1} > x_k$ and $y_{k+1} > y_k$ follow by entirely analogous proofs.

\begin{figure}
    \centering
    \includegraphics[width=8cm]{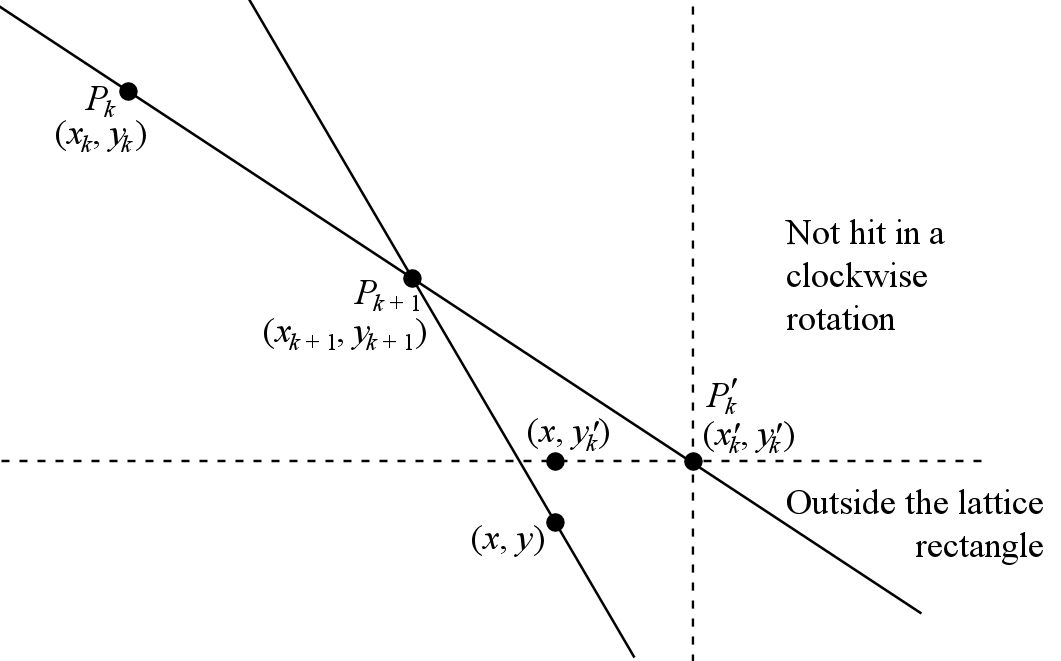}
    \caption{A clockwise rotation of $\protect\overleftrightarrow{P_kP_{k+1}}$ around $P_{k+1}$ until it reaches $(x,y)$ where $y_{k+1} < y_k$ and $x_{k+1} > x_k$ and $P_k' \notin S$. Note that if $(x,y) \in S$ then $(x,y_k')$ is also in $S$.}
    \label{xyproof}
\end{figure}

Now, we consider the cases where $x_k = x_{k+1}$ or $y_k = y_{k+1}$.

Because $\overleftrightarrow{P_kP_{k+1}} \cap S = \overline{P_kP_{k+1}} \cap S$, we know that $P_k$ and $P_{k+1}$ must be on opposite sides of the boundary of $S$. Therefore there exist no points in $S$ on $m_r$ other than $P_{k+1}$, unless $P_k$ and $P_{k+1}$ are corners of the boundary of $S$.

So suppose $P_k$ and $P_{k+1}$ are corners of the boundary of $S$. Note that $S$ is not a subset of a line, for if it were, all the regions formed by ${\mathcal L}_S$ would be unbounded (since all its lines are concurrent). Also, $P_k$ and $P_{k+1}$ do not correspond to opposite corners of the boundary of $S$ since they share a horizontal or vertical component. As we rotate $\overleftrightarrow{P_kP_{k+1}}$, the line will either rotate through $90^\circ$ or will rotate through less than $90^\circ$ before hitting another point of $S$. In the latter case, $m_r \cap S = \emptyset$, so $m_l \cap S$ must not be empty.
In the former case, $m_l \cap S = \emptyset$, and $m_r$ consists of the points of the boundary of $S$ that form the side that contains $P_{k+1}$ other than $\overleftrightarrow{P_kP_{k+1}}$.  In the former case, according to our procedure for finding the next side, $P_{k+2}$ will be the next corner of the boundary of $S$ and $P_{k+3}$ will be the final corner of the boundary of $S$ (assuming that $\Delta$ is bounded). Then the $P_i$ are among the $O_j$, and $\Delta$ contains $O$. (Note that for $\Delta$ to be a bounded region containing $O$, the origin of of the coefficient plane must be contained within the boundary of $S$.)

Thus, if $\Delta$ is bounded but does not contain $O$, there will be a point in $S$ other than $P_{k+1}$ on $m_l$, by Theorem~\ref{multiplepoints}, $P_{k+2}$ will be on $m_l$. By Proposition ~\ref{valueofd}, since $m_l$ is to the left of $\overleftrightarrow{OP_{k+1}}$, we know that $D_{k+2} = l$.

Now, consider the case where $D_k=D_{k+1}=l$. The proof that $D_{k+2}=r$ is similar to the case where $D_k=D_{k+1}=r$ (up to the point where $P_k$ and $P_{k+1}$ are corners of the boundary of $S$) with a few key differences which we point out here:

First, to find $m$, we would move $P$ from $P_k$ to the left along $M$, which corresponds to counterclockwise rotation of $m(P)$ instead of clockwise rotation.

Then, where we considered $y_{k+1} < y_k$ and $x_{k+1} > x_k$, we replace the case where both $x < x_k'$ and $y < y_k'$ with the condition that both $x > x_k'$ and $y > y_k'$. If $x > x_k'$ and $y > y_k'$, then $(x_k', y)$ is in $S$. When $m(P)$ rotates counterclockwise, $(x_k', y)$ would be hit before $(x, y)$, therefore $(x, y)$ cannot be on $m_r$.
We also replace the case where $x \ge x_k'$ and $y \ge y_k'$ with $x \le x_k'$ and $y \le y_k'$.

We now pick up the proof of the case where $D_k=D_{k+1}=l$ supposing that $P_k$ and $P_{k+1}$ are consecutive corners of the boundary of $S$. 

As we move $P$ from $P_k$ to $P^*$, $m(P)$ will rotate counterclockwise either through $90^\circ$ or through less than $90^\circ$ before hitting another point of $S$. In the former case, $m_l \cap S = \emptyset$, and $m_r$ consists of the points of the boundary of $S$ that form the side that contains $P_{k+1}$ other than $\overleftrightarrow{P_kP_{k+1}}$. 
The situation here differs from the case where $D_k=D_{k+1}=r$ in that the bounded case does not actually exist because, if $\Delta$ is bounded, the origin of $\coeffs$ must be contained within the boundary of $S$.
However, if this is true, only a rotation clockwise would cause $\overleftrightarrow{P_kP_{k+1}}$ to rotate through $90^\circ$, so $D_k=D_{k+1} \neq l$.
In the latter case, if $\Delta$ is bounded, then $m_r \cap S = \emptyset$.
\end{proof} 

\begin{theorem}\label{SandT}
The set of lines ${\mathcal L}_S$ forms no polygons with more than $4$ sides.
\end{theorem}
\begin{proof}
Let $\Delta$ be a polygon in the pattern of lines formed by ${\mathcal L}_S$ which does not contain the origin. By Lemma~\ref{numberorient}, the values $D_k$ will change only twice for any given polygon in $\mathbb{R}^2$. 

By Lemma~\ref{RRL}, if $\Delta$ is bounded and does not contain $O$, no more than two consecutive $D_k$'s can be equal, so there can be only a maximum of four $D_k$'s and hence a maximum of four sides to $\Delta$.

Now consider the polygon formed by the lines of $L_S$ that does contain the origin. By Proposition~\ref{regioncontainingo}, in a rectangular lattice, the points in $\coeffs$ that correspond to the boundary lines of the region that contains $O$ are a subset of the corners of the boundary of the lattice, so the region containing $O$ has a maximum of four sides. 
\end{proof}

\begin{remark} \label{pentagon}
We note that this result does not work for general arrays in $\coeffs$ as the set
$$S = \{(-2,-2), (-2, 2), (-2, 3), (2,-2), (2,2), (2,3), (3,-2), (3,2), (3,3)\}$$
is a $3$-by-$3$ array where ${\mathcal L}_S$ contains a pentagon. The regularity of rectangular lattices matters.
\end{remark}

\begin{proposition} \label{main2}
Let $S^{\prime}$ be the image of $S$ under a non-degenerate linear transformation. Then the set of lines ${\mathcal L}_{S^{\prime}}$ forms no polygons with more than $4$ sides.
\end{proposition}
\begin{proof}
This proposition follows from Proposition~\ref{lineartransformations}. 
\end{proof}

\section{Further Discussion} \label{conclusion}
There seem to be interesting connections to number theory which may be worth further exploration. For example, if $S = \{(a,b) \in {\mathbb{Z}}^2~|~ -n \le a,b \le n\} \cap \coeffs$, for some positive integer $n$, then the number of two-sided unbounded regions in ${\mathcal L}_S$ is equal to $\#\{(a, b) \in S ~|~ \text{$a$ and $b$ are relatively prime}\}$.

Unlike the case with a rectangular lattice, there do exist $S \subset \coeffs$ such that $D_k=D_{k+1}$ does not imply $D_{k+1}\neq D_{k+2}$ but such that the polygons formed by the lines in ${\mathcal L}_S$ are all triangles and quadrilaterals. Figure~\ref{notsaturated} shows an example. 

\begin{figure}
    \centering
    \includegraphics[width=6cm]{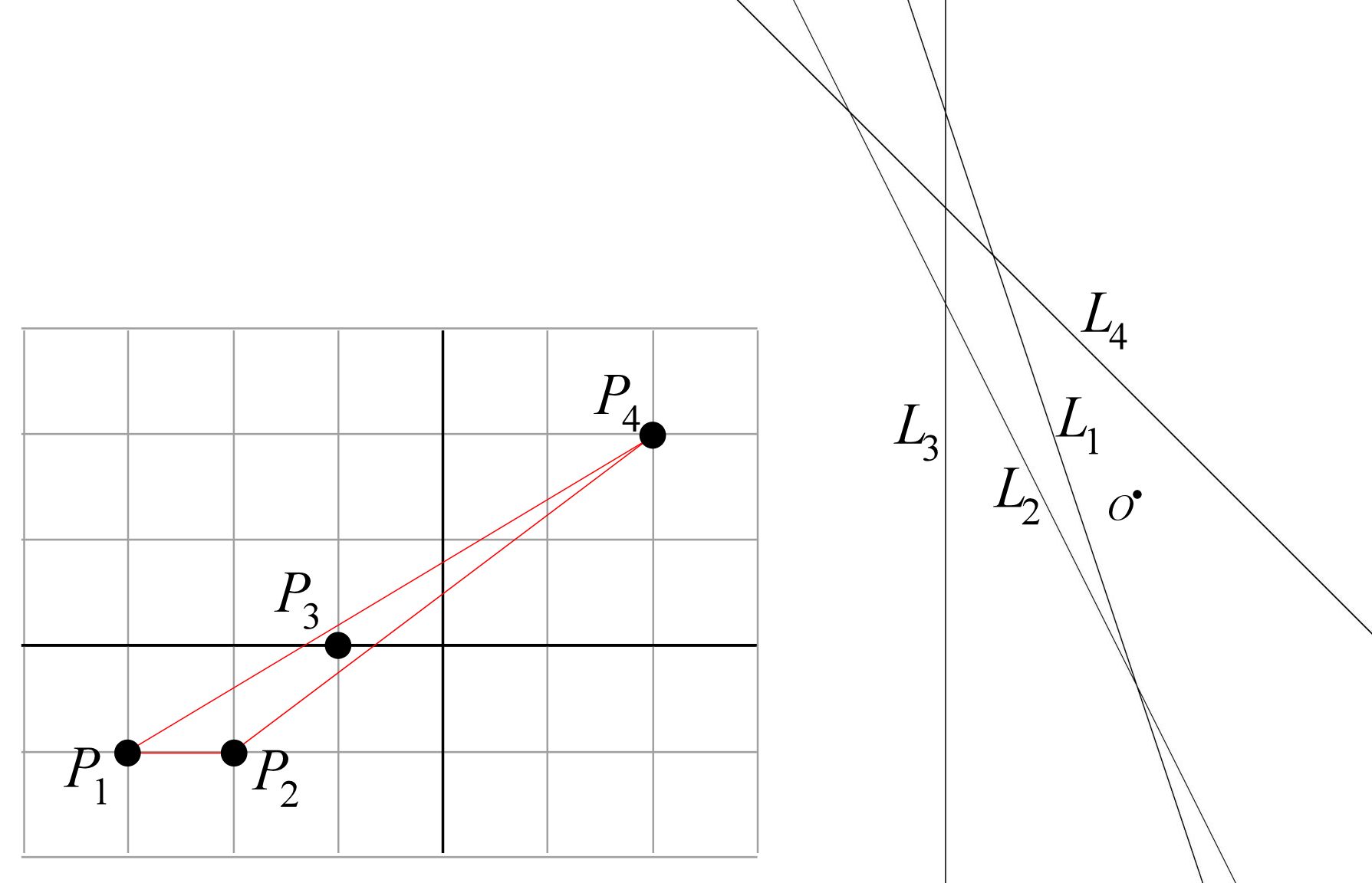}
    \caption{Lattice points on the boundary or inside the triangle whose vertices are $P_1=(-3,-1)$, $P_2=(-2,-1)$, and $P_4 = (2,2)$ and the corresponding pattern of lines. Even though these $4$ lattice points are the intersection of the lattice with a convex polygon, we have $D_2 = D_3 = D_4 = r$.}
    \label{notsaturated}
\end{figure}

Still, one might ask: if the lattice is intersected with any convex polygonal region, does the corresponding pattern of lines bound polygons that are all triangles and quadrilaterals? The answer is no, as shown by the counterexample in Figure~\ref{trianglecounterexample}. 

In fact, there is an $(n+2)$-sided polygon formed by the lines corresponding to the lattice points inside or on the boundary of the triangle whose vertices are given by $(-3, -1)$, $(-2, -1)$, and $(-3+F_{2n+1}, -1+F_{2n})$, where $F_n$ is the Fibonacci sequence with $F_1=F_2=1$. This example is based on geometrical properties of convergents of the continued fraction expansion of the golden mean. Figures~\ref{notsaturated} and \ref{trianglecounterexample} correspond to the cases $n=2$ and $n=3$, respectively.

\begin{figure}
    \centering
    \includegraphics[width=9.5cm]{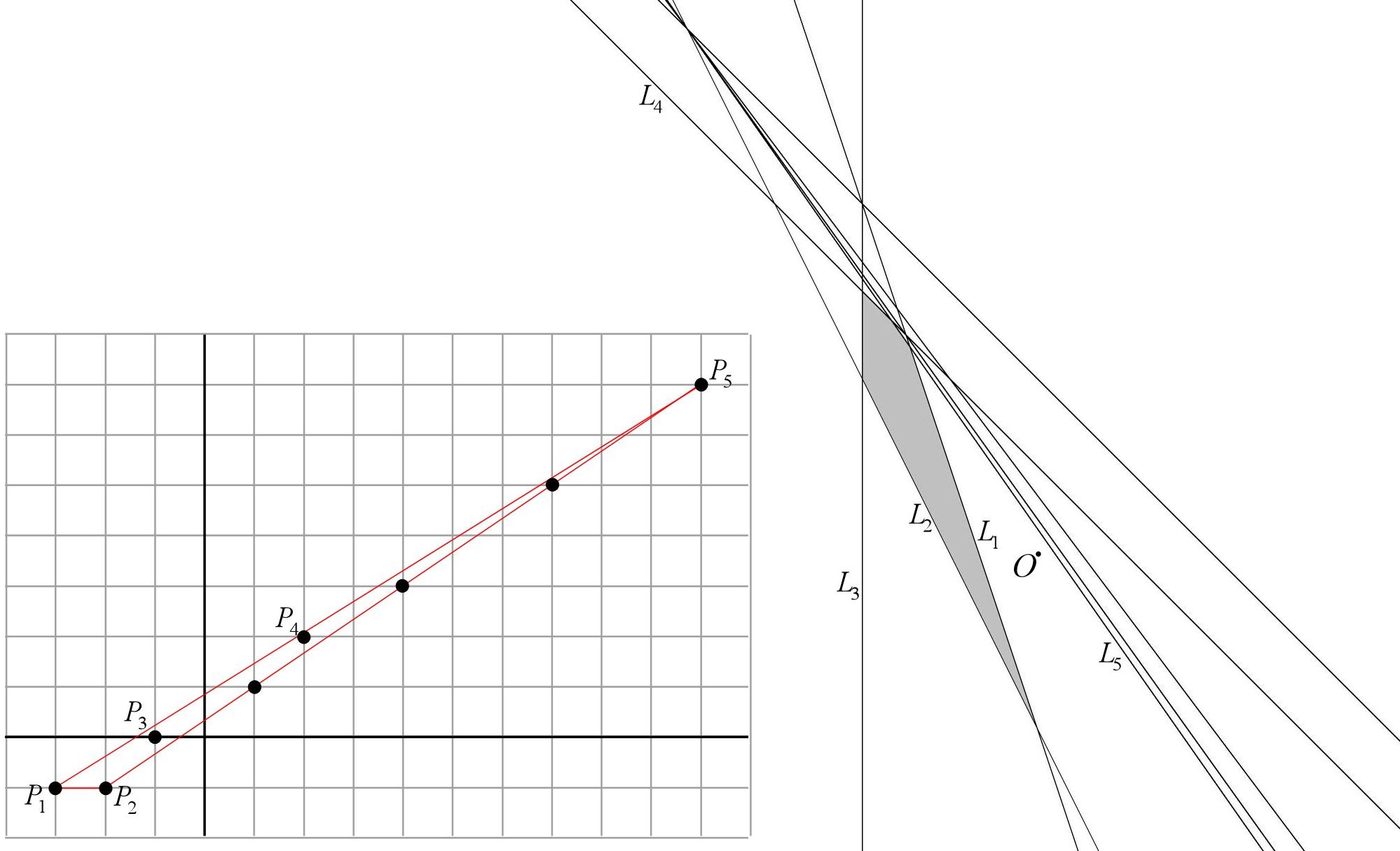}
    \caption{Lattice points on the boundary or inside the triangle with vertices $P_1 = (-3, -1)$, $P_2=(-2,-1)$, and $P_5=(10,7)$ together with the corresponding pattern of lines. In this figure, $L_k = L_{P_k}$. We get the shaded pentagon.}
    \label{trianglecounterexample}
\end{figure}

\section{Acknowledgements}
The authors would like to thank C.~Kenneth~Fan for his assistance in obtaining the results of the paper as well as helping to create and edit this paper.

\bibliographystyle{plain}
\bibliography{5+gon_free_MH_IL_arxiv_111623} 

\end{document}